\documentclass[reqno,english]{amsart}

\usepackage{etex}
\usepackage{amsmath,amssymb,amsthm,bbm,mathtools,comment}
\usepackage[shortlabels]{enumitem}
\usepackage[pdftex,colorlinks,backref=page,citecolor=blue]{hyperref}
\usepackage[mathscr]{euscript}
\usepackage{tikz,babel,adjustbox}
\usepackage{cmtiup}
\usepackage{amsfonts}
\usepackage{graphicx}
\usepackage{caption}
\usepackage{verbatim}
\usepackage{array}
\usepackage[frame,cmtip,arrow,matrix,line,graph,curve]{xy}
\usepackage{graphpap, color, pstricks}
\usepackage{pifont}
\usepackage{cmtiup}
\usepackage{amssymb}
\usepackage{amsthm}
\allowdisplaybreaks

\theoremstyle{plain}
\newtheorem{theorem}{Theorem}[section]

\newtheorem{lem}[theorem]{Lemma}

\newtheorem{prop}[theorem]{Proposition}

\theoremstyle{remark}

\renewcommand{\dots}[0]{\ldots}

\newcommand{\beq}[1]{\begin{equation}\label{#1}}
\newcommand{\enq}[0]{\end{equation}}

\usepackage{etex}
\usepackage{amsmath,amssymb,amsthm,bbm,mathtools,comment}
\usepackage[shortlabels]{enumitem}
\usepackage[pdftex,colorlinks,backref=page,citecolor=blue]{hyperref}
\usepackage[mathscr]{euscript}
\usepackage{tikz,babel,adjustbox}
\usepackage{cmtiup}
\usepackage{amsfonts}
\usepackage{graphicx}
\usepackage{caption}
\usepackage{verbatim}
\usepackage{array}
\usepackage[frame,cmtip,arrow,matrix,line,graph,curve]{xy}
\usepackage{graphpap, color, pstricks}
\usepackage{pifont}
\usepackage{cmtiup}
\usepackage{amssymb}

\allowdisplaybreaks

\theoremstyle{plain}
\theoremstyle{remark}

\renewcommand{\dots}[0]{\ldots}

\title{\textbf{\raggedright On Maximal Left-Compressed Intersecting Families Generated by a Collection of Subsets of [$n$]}}

\author{\raggedright \textnormal {Nguyen Trong Tuan }}

\date{} 
\usepackage{fancyhdr}
\fancypagestyle{firstpage}{%
  \fancyhf{}%
  \fancypagestyle{firstpage}{
  \fancyhf{}
  
  \fancyfoot[L]{\small
  Nguyễn Trọng Tuấn\\
  Viện Toán học, Viện Hàn lâm KH\&CN Việt Nam\\
  \texttt{email@example.com}}
}
  \fancyfoot[L]{\footnotesize
  Nguyen Trong Tuan\\
  Faculty of Mathematics and Computer Science, University of Science, Ho Chi Minh City, Vietnam.\\
227 Nguyen Van Cu St., Cho Quan Ward, Ho Chi Minh City.\\
  Email address: \texttt{23n21103@student.hsmus.edu.vn}.
  
\textsuperscript .}%
}
\begin{document}
\maketitle

\thispagestyle{firstpage}
\pagestyle{empty}  
\begin{flushleft}
\textbf{Abstract.} \,
We provide a characterization of maximal left-compressed families based on their generating sets $\mathcal{G}\subseteq 2^{[n]}$. We show that there is a one-to-one correspondence between maximal left-compressed families $\mathcal{A}\subseteq \binom{[n]}{k}$ and principal generating sets. Moreover, we give a complete description of maximal left-compressed intersecting families having exactly two maximal generators. Based on this, for any $X\subseteq [2,n]$, we compare $\mathcal{A}(X)$, where $\mathcal{A}$ is a rank-2 maximal left-compressed intersecting family with exactly two maximal generators, and $\mathcal{S}(X)$, where $\mathcal{S}$ denotes the Star family.
\end{flushleft}



\section{Introduction}

Throughout this paper, a subset $A$ of $[n]$ consisting of $r$ elements is called an $r$-set and always is listed in ascending order: $A=\{a_{1},a_{2},\ldots,a_{r}\}$, where $1\leq a_{1}< a_{2}<\ldots<a_{r} \leq n$. For each positive integer $n$, we denote $[n]=\{1,2,\ldots,n\}$ and let $m \leq n$ be positive integers, we denote $[m,n]= \{m,m+1,\ldots,n\}$. \\ \\
In what follows, we assume that $n$ and $k$ are positive integers satisfying $4 \leq 2k \leq n$. We denote $\binom{[n]}{k}$ as the collection of all $k$-sets of $[n]$. A family $\mathcal{A}$ is $k$-uniform if every member of $\mathcal{A}$ contains exactly $k$ elements. A family $\mathcal{A} \subseteq 2^{[n]}$ is called \textit{ intersecting } if for all $S,T \in \mathcal{A}$, we have $S \cap T \neq \emptyset$. An intersecting family $\mathcal{A}$ is \textit{trivial} if $\bigcap_{A \in \mathcal{A}}A \neq \emptyset$. The families $\mathcal{A}$ and $ \mathcal{B}$ are said to be \textit{cross-intersecting} if any $A \in \mathcal{A}$ and $B \in \mathcal{B}$, we have $A \cap B \neq \emptyset$. Given two sets $A,B \in \binom{[n]}{k}$ with $A=\{a_{1},a_{2},\ldots,a_{k} \}$ and $B=\{b_{1},b_{2},\ldots,b_{k}\}$, we define $A \leq B$ if $a_{i} \leq b_{i}$ for all $i=1,2,\ldots,k$. We also define $A < B$ if $A \leq B$ and there exists $1 \leq i \leq k$ such that $a_{i} < b_{i}$. We say that two sets $A$ and $B$ (not necessarily distinct) are \textit{strongly intersecting} and write $A \sim_{si} B$ if for any $A',B'$ such that $A'\leq A, B'\leq B$ then $A' \cap B' \neq \emptyset$. A family is said to be strongly intersecting if every pair of its members is strongly intersecting.
A family $\mathcal{A} \subseteq \binom{[n]}{k}$ is said to be \textit{left-compressed} if for every $A \in \mathcal{A}$ and for every $ B$ such that $B \leq A$, we also have $B \in \mathcal{A}$. In fact, left-compressed families and the compression technique (see \cite{Fr87} for a survey) are powerful tools in extremal set theory. For $A \in 2^{[n]}$, we define $\mathcal{L}(A)=\left\{S \in \binom{[n]}{|A|}: S \leq A\right\}$. It is obvious that $\mathcal{L}(A)$ is left-compressed and we say that $\mathcal{L}(A)$ is the \textit{left-compression
closure} of $A$. A $k$-set $A$ in the left-compressed family $\mathcal{A} \subseteq \binom{[n]}{k}$ is \textit{maximal set in $\mathcal{A}$}   if there does not exist any $B \in \mathcal{A}$ such that $A < B$. For two sets $A=\{a_{1},a_{2},\ldots,a_{r}\} \in 2^{[n]}$ and $B=\{b_{1},b_{2},\ldots,b_{s} \} \in 2^{[n]}$, we define $A \preceq B$ if  $r \geq s$ and $a_{i} \leq b_{i}$ for all $1 \leq i \leq s$. Note that the order relations "$\leq$" on $\binom{[n]}{k}$ and "$\preceq$" on $2^{[n]}$ are both partial orders. Therefore, every nonempty subset of $\binom{[n]}{k}$ (or $2^{[n]}$) has at least one maximal element.\\ \\
We refer to a family that is both intersecting and left-compressed as a  \textit{left-compressed intersecting family }, abbreviated as LCIF. It is obvious that if $\mathcal{A}$ is LCIF and $A,B \in \mathcal{A}$ then $A \sim_{si}B$, equivalently, $\mathcal{F}({A})$ and $\mathcal{F}(B)$ are cross-intersecting. Finally, we say an intersecting family $\mathcal{A} \subseteq \binom{[n]}{k}$  is \textit{maximal} if no other set can be added to $\mathcal{A}$ while preserving the intersecting property. There are some well-known MLCIF families: the \textit{Star family} $S=\{ A\in \binom{[n]}{k}: 1\in A\}$; the family $\mathcal{A}_{2,3}=\{A \in \binom{[n]}{k}:|A\cap \{1,2,3\}|\geq 2\}$ and the \textit{Hilton-Milner family} $\mathcal{H}\mathcal{M}=\{A\in \mathcal{S}:A\cap \{2,\ldots,k+1\} \neq \emptyset \}\cup [2,k+1]$.
\\ \\
 To study LCIFs, we need a suitable way to describe left-compressed families. We find that if $A \preceq G$ and $B \leq A$ then $B \preceq G$. This observation provides the basis for defining a left-compressed family over $\binom{[n]}{k}$. This definition is a modification of that of Ahlswede and Khachatrian \cite{AK}. 
Let $\mathcal{G}$ be a collection of subsets of $[n]$. We define a $k$-uniform family as follows:
    $$\mathcal{F}(n,k,\mathcal{G})=\bigg \{S \in \binom{[n]}{k} : S \preceq G \text{ for some } G \in \mathcal{G}  \bigg \}$$
It is clear that if there exists $G \in \mathcal{G}$ such that $|G| > k$ then there does not exist $k$-set such that $S \preceq G$. Therefore, we assume that $|G| \le k$, for every $G \in \mathcal{G}$. When there is no risk of confusion, we may write $\mathcal{F}(\mathcal{G})$ instead of $\mathcal{F}(n,k,\mathcal{G})$. In the special case  when $\mathcal{G}=\{G\}$ with $G=\{g_{1},\ldots,g_{r}\}$, we simply write $\mathcal{F}(G)$ or $\mathcal{F}(g_{1},\ldots,g_{r})$ instead of $\mathcal{F}(\{G\})$. If $|G|=k$ then $\mathcal{F}(G)$ is the set of all $k$-set $S$ satisfying $S \leq G$ and in this case, we have $\mathcal{L}(G)=\mathcal{F}(G)$. We see that $\mathcal{F}(\mathcal{G})$ is left-compressed and $\mathcal{F}(\mathcal{G})=\bigcup_{G \in \mathcal{G}}\mathcal{F}(G)$. It is clear that not every $\mathcal{F}(\mathcal{G})$ is an intersecting family. For example, the families $\mathcal{F}(\{\{2,3\}\}),\mathcal{F}(\{\{1,3\},\{1,4,5\},\{2,3,5\} \}$ are intersecting, while the families $\mathcal{F}(\{\{2,4\}\}), \mathcal{F}(\{\{2,3\},\{2,4,5\} \}$ are not intersecting. Therefore, a natural question arises: under what conditions on $\mathcal{G}$, the $\mathcal{F}(\mathcal{G})$ is an intersecting family? In \cite{TTT}, the authors provided a necessary and sufficient condition on $\mathcal{G}$ for $\mathcal{F}(\mathcal{G})$ to be an intersecting family.
\begin{theorem}(Tuan Nguyen-Thi Nguyen-Thu Tran \cite{TTT}) 
Let $\mathcal{G} \subseteq  2^{[n]}$. The family $\mathcal{F}(\mathcal{G})$ is intersecting if and only if for every $G, H \in \mathcal{G}$ (possibly identical), there exists an $1 \le \ell \le n$ such that $\mu_{G}(\ell) + \mu_{H}(\ell) > \ell$. 
\end{theorem}
Thus, $G \sim_{si} H$ if and only if there exists an integer $1 \le \ell \le n$ such that $\mu_{G}(\ell) + \mu_{H}(\ell) > \ell$. In this case, we say that $G$ and $H$ are strongly intersecting at $\ell$ and we write $G \sim_{si} H$ at $\ell$. Conversely, we say that $G$ and $H$ are \emph{not strongly intersecting} if, for all $1 \le \ell \le n$, we have $\mu_{G}(\ell) + \mu_{H}(\ell) \le \ell$; equivalently, there exist $S \le G$ and $T \le H$ such that $S \cap T = \emptyset$.\\
Our aim is to characterize MLCIFs, which in turn enables the construction such families. In \cite{B}, Baber constructed MLCIFs by extending each maximal family on $[2k]$ to a maximal family on $[n]$.

\begin{theorem}(Baber \cite{B}) Let $\mathcal{A}\subseteq \binom{[2k]}{k}$ be a maximal left-compressed intersecting family. Then, $\mathcal{A}$ extends uniquely to a maximal left-compressed intersecting subfamily of $\binom{[n]}{k}$. Moreover, every maximal left-compressed intersecting subfamily of $\binom{[n]}{k}$ arises in this way. 
\end{theorem}

In this paper, we present an alternative way to construct MLCIFs based on their generating sets. We introduce the following definitions. First, consider the sets $A^{1},\ldots,A^{m} \in 2^{[n]}$. For each $j \in [m]$, we denote $r_{j}=|A^{j}|$ and $d=\max\{r_{j}\}$. Suppose $A^{j}=\{a^{j}_{1},\ldots,a^{j}_{r_{j}}\}$. We define the set $$A=A^{1}\wedge A^{2}\wedge \ldots \wedge A^{m}=\{a_{1},\ldots,a_{d}\}$$
where $a_{i}=min\{a^{j}_{i}: 1 \leq j \leq m\}$ with convention $a^{j}_{i}=+\infty$ if $i > r_{j}$\\
Next, as we shall prove in Lemma 2.4, if $\mathcal{F}(\mathcal{G})$ is an MLCIF, then every $G \in \mathcal{G}$ satisfies $|G \cap [k+|G|-1]|=|G|$; in other words, if $G=\{g_{1},\ldots,g_{r}\}$, then $g_{r} \le k+r-1$. This motivates the following definition.
$$G_{k}=\left \{G \in \binom{[2k-1]}{\leq k}: |G \cap [k+|G|-1]|=|G| \right \}$$
We also define $\mathcal{G}_{k}(1)=\{ G \in \mathcal{G}_{k}: 1 \in G\}$ and $\mathcal{G}_{k}(\overline{1})= \{ G \in \mathcal{G}_{k}: 1 \notin G\}$.
We call a subfamily $\mathcal{G} \subseteq \mathcal{G}_{k}(\overline{1})$ a \emph{principal generating set } if $\mathcal{G}$ are strongly intersecting.  (abbreviated as PGS ). Now, let $\mathcal{G}=\{G^{1},\ldots,G^{m}\} \subseteq \mathcal{G}_{k}(\bar{1})$ be an PGS. For every $j \in [m]$, let $r_{j} = |G^{j}|$ and write 
$G^{j} = \{ g^{j}_{1}, \ldots, g^{j}_{r_{j}} \}$. For each $i_{j} \in [r_{j}]$, we define 
$G^{j}_{i_{j}} = \{1\} \cup [i_{j}+1,\, g^{j}_{i_{j}}]$. Finally, for $(i_{1},\ldots,i_{m}) \in ([r_{1}]\times \ldots \times [r_{m}]]$, we define
$$\mathcal{H}(\mathcal{G})=\left\{H: H \le  G^{1}_{i_{1}}\wedge\cdots \wedge G^{m}_{i_{m}}, \text{for some } (i_{1},\ldots,i_{m}) \in [r_{1}]\times \ldots \times [r_{m}] \right\} $$
We can now state our main result. 
\begin{theorem} Let $\mathcal{G}$ be a PGS. Then the family $\mathcal{F}(\mathcal{G}) \cup \mathcal{F}(\mathcal{H}(\mathcal{G}))$ is an MLCIF.
Conversely, for every MLCIF $\mathcal{A} \subseteq \binom{[n]}{k}$, there exists a PGS $\mathcal{G}$ such that $\mathcal{A}=\mathcal{F}(\mathcal{G}) \cup \mathcal{F}( \mathcal{H}(\mathcal{G}))$.
\end{theorem}
Theorem 1.3 allows us to conclude that there is a one-to-one correspondence between the sets of PGS $\mathcal{G}$ and the sets of MLCIF $\mathcal{A} \subseteq \binom{[n]}{k}$. This means that the number of MLCIFs is exactly the same as the number of PGSs $\mathcal{G} \subseteq \mathcal{G}_{k}(\overline{1})$. Theorem 1.3 further states that for any MLCIF  $\mathcal{A} \subseteq \binom{[n]}{k}$, there exists a unique PGS $\mathcal{G} \subseteq \mathcal{G}_{k}(\overline{1})$ such that $\mathcal{A}=\mathcal{F}(\mathcal{G})\cup \mathcal{F}(\mathcal{H}(\mathcal{G}))$. We say that $\mathcal{A}$ \textit{is determined by $\mathcal{G}$} and \textit{generating set } of $\mathcal{A}$ is precisely $\mathcal{G}\cup \mathcal{H}(\mathcal{G})$. Each $G \in \mathcal{G}\cup \mathcal{H}(\mathcal{G})$ is called a \textit{generator } of $\mathcal{A}$. A generator $G$ of a generating set $\mathcal{G}$ is said to be \emph{ maximal in $\mathcal{G}$} if there is no $G' \in \mathcal{G}$, $G' \ne G$, such that $G \preceq G'$.
We say that an MLCIF $\mathcal{A}$ has $r$ maximal generators if the family $\mathcal{G}\cup \mathcal{H}(\mathcal{G})$ has exactly $r$ maximal generators. We also define the \textit{rank } of $\mathcal{A}$ as the minimum size of a generator of $\mathcal{A}$.\\ \\
Another topic addressed in our paper is the problem of intersecting a family of sets with a given set $X$. For a family $\mathcal{A} \subseteq 2^{[n]}$ and a set $X \subseteq [n]$, we define 
\[
\mathcal{A}(X) = \{ A \in \mathcal{A} : A \cap X \neq \emptyset \}
\quad \text{and} \quad
\mathcal{A}_{0}(X) = \{ A \in \mathcal{A} : A \cap X = \emptyset \}.
\]
We know that one of the most fundamental theorems on intersecting families is the Erdős–Ko–Rado theorem.
\begin{theorem} (Erd\H{o}s-Ko-Rado \cite{EKR}) Let $n \geq 2k$ and let $\mathcal{A} \subseteq\binom{[n]}{k}$ be an intersecting family. Then $|\mathcal{A}| \leq |\mathcal{S}|$.
\end{theorem}
In \cite{PB}, Borg considered a variant of the Erdős–Ko–Rado theorem.  Borg asked which $X \subseteq [2,n]$ have the property that $|\mathcal{A}(X)| \leq |\mathcal{S}(X)|$ for all MLCIFs $\mathcal{A} \subseteq \binom{[n]}{k}$. Since the Star family has a simple structure, for every $X \subset [2,n]$ with $|X|=d$, it is easy to compute that $|\mathcal{S}(X)|=\binom{n-1}{k-1}-\binom{n-d-1}{k-1}$. There are some investigations concerning the comparison between $\mathcal{A}(X)$ and $\mathcal{S}(X)$. There have been several previous studies on this problem \cite{PB,B,BB,BM}. In \cite{BB}, Bond determined $\mathcal{A}(X)$ in the case when $\mathcal{A}$ is an MLCIF with exactly one maximal generator. The analogous problem in the case when $\mathcal{A}$ has exactly two maximal generators is more complicated. However, we also establish the following result under the additional assumption that $\mathcal{A}$ has rank 2. First, we prove the following theorem in order to characterize the MLCIFs of rank 2 with exactly two maximal generators.
\begin{theorem}
Let $\mathcal{A}=\mathcal{F}(\mathcal{G})$ be an MLCIF with two maximal generators. Then
\[
\mathcal{A}=\mathcal{F}\big([a,b]\big)\cup \mathcal{F}\big(\{1\}\cup[b-a+2,b]\big),
\]
where $\mathcal{G}=\{[a,b],\,\{1\}\cup[b-a+2,b]\}$ with $b>2a-1$.
Furthermore, if $\mathcal{A}$ has rank $2$ then
\[
\mathcal{A}=\mathcal{F}(\{[2,b]\})\cup \mathcal{F}(\{1,b\}),
\]
where $4\le b\le k+1$.
\end{theorem}
Next, we present the following result concerning the intersection of a rank-$2$ MLCIF with exactly two maximal generators and a set $X \subseteq [2,n]$. Although our results do not cover general rank-2 MLCIFs, they yield explicit expressions for the measure of $\mathcal{A}(X)$. See \cite{BM} for related discussions. 
\begin{theorem} Let $\mathcal{A} = \mathcal{F}([2,b]) \cup \mathcal{F}(1,b)$ with $4 \le b \le k+1$ be an MLCIF of rank $2$ having exactly two maximal generators, and let $X \subseteq [2,n]$ with $|X| = d$. Then:

\begin{enumerate}[(i)]
  \item If $X \cap [2,b] \neq \emptyset$, then
  \[
  |\mathcal{A}(X)| =
  \binom{n-1}{k-1}
  - \binom{n-b}{k-1}
  + \binom{n-b}{k-b+1}
  - \binom{n-d-1}{k-1}
  + \binom{n-d-b+\mu_X(b)}{k-1}.
  \]
  In particular, if $X \subseteq [2,b]$, then
  \[
  |\mathcal{A}(X)| = |\mathcal{S}(X)| + \binom{n-b}{k-b+1};
  \]
  and if $X \cap [2,b] \subsetneq X$, then for sufficiently large $n$ we have
  \[
  |\mathcal{A}(X)| < |\mathcal{S}(X)|.
  \]

  \item If $X \cap [2,b] = \emptyset$, then
  \[
  |\mathcal{A}(X)| =
  \binom{n-1}{k-1}
  - \binom{n-b}{k-1}
  + \binom{n-b}{k-b+1}
  - \binom{n-d-1}{k-1}
  + \binom{n-d-b}{k-1}
  - \binom{n-d-b}{k-b+1}.
  \]
  Consequently,
  \[
  |\mathcal{A}(X)| \le |\mathcal{S}(X)|.
  \]
\end{enumerate}
\end{theorem}

Our paper is divided into four main sections. In the first section, we present fundamental concepts and outline the main content of the paper. In the second part, together with some auxiliary lemmas, we present a proof of Theorem 1.3. In the next section,  we will present the proof of Theorem 1.5 and Theorem 1.6 as well as describe some special MLCIFs rank 2. Finally, in the last section, we propose several open questions concerning MLCIFs related to generating sets $\mathcal{G}$.
\section{Proof of Theorem 1.3}
\subsection{Preliminaries}
For any $X \subseteq [n]$ and $1 \le \ell \le n$, we denote by $\mu_{X}(\ell)$ the number of elements of the set $X$ not exceeding $\ell$, that is, $\mu_{X}(\ell)=|X \cap [\ell]|$. Note that $\mu_{X}(\ell)$ is a non-decreasing function of $\ell$. For $\ell \ge 2$, we have $\mu_{X}(\ell) - \mu_{X}(\ell-1) \in \{0,1\}$. The following are some simple properties of $\mu_{X}(\ell)$ with respect to two arbitrary sets $A$ and $B$.

\begin{prop}
For non-empty subsets $X, A, B \subseteq [n]$ and for all $ 1 \le \ell \le n$, we have:
\begin{enumerate}[(i)]
\item If $A \preceq B$, then $\mu_{A}(\ell) \ge \mu_{B}(\ell)$;
\item If $A \le B$, then $\mu_{A}(\ell) \ge \mu_{B}(\ell)$;
\item If $A \subseteq B$, then $\mu_{B}(\ell) \ge \mu_{A}(\ell)$.
\end{enumerate}
\end{prop}

\begin{proof}
(i) Assume that $A = \{a_{1}, a_{2}, \ldots, a_{r}\}$ and $B = \{b_{1}, b_{2}, \ldots, b_{s}\}$.  
Since $A \preceq B$, we have $r \ge s$ and $\{a_{1}, a_{2}, \ldots, a_{s}\} \le \{b_{1}, b_{2}, \ldots, b_{s}\}$.

Now, we consider different cases for $ 1 \le \ell \le n$.  

If $\ell \ge b_{s}$, then $\mu_{B}(\ell) = s$. Since $a_{s} \le b_{s} \le \ell$, we also have $\mu_{A}(\ell) \ge s = \mu_{B}(\ell)$.  

If $\ell < b_{1}$, then $\mu_{B}(\ell) = 0$, which proves the required result.  

Finally, consider the case $b_{1} \le \ell < b_{s}$. Suppose $|B \cap [\ell]| = p$.  
This means that there are $p$ elements $b_{1}, \ldots, b_{p}$ not exceeding $\ell$.  
Since $a_{1} \le b_{1}, \ldots, a_{p} \le b_{p}$, we conclude that there are at least $p$ elements $a_{1}, \ldots, a_{p}$ not exceeding $\ell$.  
Thus, $\mu_{A}(\ell) \ge p = \mu_{B}(\ell)$.  

(ii) Since $A \le B$, we have $A \preceq B$. Hence, by part (i), $\mu_{A}(\ell) \ge \mu_{B}(\ell)$.  

(iii) Assume that $|A| = r$ and $|B| = s$. Since $A \subseteq B$, we have $r \le s$.  
We can observe that the first $r$ elements of $B$ do not exceed the corresponding $r$ elements of $A$, respectively.  
Hence, $B \preceq A$. Applying part (i), we also obtain $\mu_{B}(\ell) \ge \mu_{A}(\ell)$ for all $ 1 \le \ell \le n$.
\end{proof}
Theorem~1.1 (see~\cite{TTT}) characterizes when two sets are strongly intersecting. 
To further clarify the strongly intersecting property of two sets, we present the following lemma.

\begin{lem}
Let $G,H \in 2^{[n]}$. 
If $G \sim_{si} H$, then there exist indices $1 \le p \le |G|$ and $1 \le q \le |H|$ such that $g_{p} = h_{q} = p + q - 1$. In the case $G = H$, there exists $1 \le p \le |G|$ such that $g_{p} = 2p - 1$.
\end{lem}
\begin{proof}
In~\cite{TTT}, we proved that if $G \sim_{si} H$, then there exists a number $1 \le \ell \le n$ such that $\mu_{G}(\ell) + \mu_{H}(\ell) > \ell$. 
We will show that there exist indices $1 \le p \le |G|$ and $1 \le q \le |H|$ such that $g_{p} = h_{q} = p + q - 1$. 

Indeed, choose $\ell$ to be the smallest such integer. 
If $\ell = 1$, then $g_{1} = h_{1} = 1$ and we can take $p = q = 1$. 
Now consider $\ell > 1$. 
Since $\ell$ is chosen to be minimal, we have
\[
\mu_{G}(\ell-1) + \mu_{H}(\ell-1) \le \ell - 1 < \ell < \mu_{G}(\ell) + \mu_{H}(\ell).
\]
Furthermore, note that $\mu_{G}(\ell) - \mu_{G}(\ell-1) \in \{0,1\}$ and $\mu_{H}(\ell) - \mu_{H}(\ell-1) \in \{0,1\}$. 
From these observations, we deduce the equalities
\[
\mu_{G}(\ell) = \mu_{G}(\ell-1) + 1, \quad 
\mu_{H}(\ell) = \mu_{H}(\ell-1) + 1, \quad 
\mu_{G}(\ell-1) + \mu_{H}(\ell-1) = \ell - 1, \quad 
\mu_{G}(\ell) + \mu_{H}(\ell) = \ell + 1.
\]
Since $\mu_{G}(\ell) + \mu_{H}(\ell) = \mu_{G}(\ell-1) + \mu_{H}(\ell-1) + 2$, we conclude that $\ell \in G \cap H$. 
Thus, there exist indices $p,q$ such that $g_{p} = h_{q} = \ell$. 
We have $\mu_{G}(\ell) = p$ and $\mu_{H}(\ell) = q$, so $\ell = p + q - 1$. 
Hence, $g_{p} = h_{q} = p + q - 1$. It is obvious that if $G=H$ then $p=q$ and $g_{p}=2p-1$.
\end{proof}

The following lemma provides several properties of a generator in a left-compressed intersecting family, which will be used in the proof of Theorem~1.3.
\begin{lem}
Let $G = \{g_{1}, \ldots, g_{r}\} \in \mathcal{G}_{k}(\overline{1})$. Then:
\begin{enumerate}[(i)]
\item If $G$ is a generator of an MLCIF, then $|G \cap [k + |G| - 1]| = |G|$;
\item Let $S = \{s_{1}, \ldots, s_{k}\}$ with $s_{1} \ge 2$ and $s_{i} \ge g_{i} + 1$ for some $i \in [r]$. Then $S$ and $G_{i} = \{1\} \cup [i + 1, g_{i}]$ are not strongly intersecting.
\end{enumerate}
\end{lem}
\begin{proof}
(i) Suppose that $G$ is a generator of an MLCIF $\mathcal{A}$. Assume, to the contrary, that $|G \cap [k + |G| - 1]| < |G|$. This means that $g_{r} \ge k + r$.  

Let $H$ be a generator of $\mathcal{A}$ (it is possible that $H = G$).  
Since $G \sim_{si} H$, there exists $\ell$ such that $\mu_{G}(\ell) + \mu_{H}(\ell) > \ell$.  
For $\ell \ge k + r$, we have $\mu_{G}(\ell) \le r \le \ell - k$.  
Hence $\mu_{G}(\ell) + \mu_{H}(\ell) \le k + \ell - k = \ell$.  
Therefore, we must have $\ell \le k + r - 1 < g_{r}$.  

Set $G' = \{g_{1}, \ldots, g_{r-1}\}$.  
Since $\ell \le k + r - 1$, we have $|G \cap [\ell]| = |G' \cap [\ell]|$,  
so $\mu_{G}(\ell) = \mu_{G'}(\ell)$.  
Thus, $\mu_{G'}(\ell) + \mu_{H}(\ell) > \ell$.  
By Theorem~1.1, $G' \sim_{si} H$.  

On the other hand, from $G' \subsetneq G$ we get $G \preceq G'$, and therefore $\mathcal{F}(G) \subsetneq \mathcal{F}(G')$.  
The family $\mathcal{G}' = (\mathcal{G} \setminus \{G\}) \cup \{G'\}$ is strongly intersecting, so $\mathcal{F}(\mathcal{G}')$ is an LCIF.  
However, $\mathcal{F}(\mathcal{G}) \subsetneq \mathcal{F}(\mathcal{G}')$, contradicting the maximality of $\mathcal{A}$.  

(ii) Now consider $S$ such that $1 \notin S$ and $s_{i} \ge g_{i} + 1$ for some $i \in [r]$.  
We will show that $\mu_{S}(\ell) + \mu_{G_{i}}(\ell) \le \ell$ for all $1 \le \ell \le n$.  
We consider the following cases:

\textbf{Case 1: $\ell = 1$.}  
We have $\mu_{S}(\ell) = 0$ and $\mu_{G_{i}}(1) = 1$, so $\mu_{S}(1) + \mu_{G_{i}}(1) = 1$.

\textbf{Case 2: $2 \le \ell \le i$.}  
We have $\mu_{G_{i}}(\ell) = 1$.  
Suppose that $s_{\ell - 1} < \ell$. Then $\ell - 1 \le s_{\ell - 1} < \ell$, implying $s_{\ell - 1} = \ell - 1$ and hence $s_{1} = 1$, a contradiction.  
Thus $s_{\ell - 1} \ge \ell$, and so $\mu_{S}(\ell) \le \ell - 1$.  
Hence $\mu_{S}(\ell) + \mu_{G_{i}}(\ell) \le \ell$.

\textbf{Case 3: $i + 1 \le \ell \le g_{i}$.}  
We have $\mu_{G_{i}}(\ell) = \ell - i + 1$.  
On the other hand, since $s_{i} \ge g_{i} + 1 \ge \ell + 1$, we have $\mu_{S}(\ell) \le i - 1$.  
Thus $\mu_{S}(\ell) + \mu_{G_{i}}(\ell) \le (i - 1) + (\ell - i + 1) = \ell$.

\textbf{Case 4: $g_{i} + 1 \le \ell < s_{k}$.}  
We have $\mu_{G_{i}}(\ell) = g_{i} - i + 1$.  
If $g_{i} + 1 \le \ell \le s_{i}$, then $\mu_{S}(\ell) \le i - 1$, giving  
$\mu_{S}(\ell) + \mu_{G_{i}}(\ell) \le i - 1 + g_{i} - i + 1 = g_{i} < \ell$.  
Otherwise, if $g_{i} + 1 \le s_{i} < \ell$, then $\mu_{S}(\ell) = j \ge i$.  
Since $g_{i} < g_{i} + 1 \le s_{i} \le s_{j} \le \ell$, we have  
$\ell - (g_{i} + 1) \ge s_{j} - s_{i} \ge j - i$.  
It follows that $\mu_{S}(\ell) + \mu_{G_{i}}(\ell) = j + g_{i} - i + 1 \le \ell$.

\textbf{Case 5: $\ell \ge s_{k}$.}  
We have $s_{k} - k \ge s_{i} - i \ge g_{i} - i + 1 = |G_{i}|$.  
Hence $\ell \ge s_{k} \ge k + |G_{i}|$.  
Therefore $\mu_{G_{i}}(\ell) = |G_{i}| = g_{i} - i + 1$, and  
$\mu_{S}(\ell) + \mu_{G_{i}}(\ell) \le k + |G_{i}| \le \ell$.  

In all cases, $\mu_{S}(\ell) + \mu_{G_{i}}(\ell) \le \ell$ for all $1 \le \ell \le n$.  
By Theorem~1.1, $S$ and $G_{i}$ are not strongly intersecting.
\end{proof}

The following lemma will also be used in the proof of Theorem 1.3.

\begin{lem}
Suppose that $A, B$ and $C$ are subsets of $[n]$. Then
\begin{itemize}
    \item[(i)] If $B \preceq A $, then $B \wedge A = B$;
    \item[(ii)] $C \preceq A \wedge B$ if and only if $C \preceq A$ and $C \preceq B$.
\end{itemize}
\end{lem}
\begin{proof}
(i) Assume $A=\{a_{1},\ldots,a_{r}\}$ and $B=\{b_{1},\ldots,b_{s}\}$. Since $B \preceq A$, we have $ r \ge s$ and $b_{i} \le a_{i}$ for all $1 \le i \le r $. If $r=s$ then it is clear that $B \wedge A = B$. If $s > r$ then $a_{i}= +\infty$ for all $r+1 \le i \le s$. So, we also have $B \wedge A = B$.\\
(ii) Let $r=|A|, s=|B|$ and $c_{i}$ be the $i$-th coordinate of $C$. Suppose that $C \preceq A \wedge B$. We have $|C| \ge |A \wedge B|$. So $|C| \ge |A|$ and $|C| \ge |B|$. We find that the first $r$ elements of $C$ do not exceed the corresponding elements of $A \wedge B$. It follows that the first $|A|$ elements of $C$ do not exceed those of $A$. Hence $C \preceq A$. Similarly, we also have $C \preceq B$.\\
We now consider the converse direction. First, since $C \preceq A$ and $C \preceq B$, we have $|C| \ge |A|$ and $|C| \ge |B|$. Moreover, $\{c_{1},\ldots,c_{r}\} \le A$ and  $\{c_{1},\ldots,c_{s}\} \le B$. So $|C| \ge |A \wedge B|$. We consider two cases as follows. First, assume that $r=s$. It is obvious that $C \le A\wedge B$. Thus, $C \preceq A \wedge B$. Next, we consider $r < s$. We have $|A \wedge B|=s$. On the other hand, $a_{i}=+\infty$ for all $i \in [r+1,s]$. So, $A\wedge B= \{m_{1},\ldots,m_{r},b_{r+1},\ldots,b_{s}\}$, where $m_{i}=\min \{a_{i},b_{i}\}$ for all $i \in [r]$. Thus, $c_{i} \le m_{i}$ for all $i \in [r]$ and $c_{i} \le b_{i}$ for all $i \in [r+1,s]]$. Hence, $\{c_{1},\ldots,c_{s}\} \le A \wedge B$ and therefore $C \preceq A \wedge B$.
\end{proof}
\subsection{ Proof of Theorem 1.3}
\begin{proof} The proof of Theorem 1.3 is divided into several parts.\\ \\
\textbf{The family $\mathcal{F}(\mathcal{G})\cup \mathcal{F}(\mathcal{H}(\mathcal{G}))$ is LCIF }\\
First, we will prove that the family $\mathcal{G}\cup \mathcal{H}(\mathcal{G})$ is strongly intersecting. Clearly, both $\mathcal{G}$ and $\mathcal{H}(\mathcal{G})$ are strongly intersecting families. Next, take any $G^{j} \in \mathcal{G}$ and $H \in \mathcal{H}(\mathcal{G})$. Then, there exists $(i_{1},\ldots,i_{m}) \in ([r_{1}] \times \ldots [r_{m}])\}$ such that $H \leq G^{1}_{i_{1}} \wedge \cdots \wedge G^{m}_{i_{m}}$. This implies that $H \preceq G^{1}_{i_{1}} \wedge \cdots \wedge G^{m}_{i_{m}}$. From Lemma 2.3, we have $H \preceq G^{j}_{i_{j}}$ for all $j \in [m]$. We have $G^{j} \sim_{si} G^{j}_{i_{j}}$ at $g^{j}_{i_{j}}$ since $\mu_{G^{j}}(g^{j}_{i_{j}})+\mu_{G^{j}_{i_{j}}}(g^{j}_{i_{j}})=i_{j}+g^{j}_{i_{j}}-i_{j}+1 >g^{j}_{i_{j}}$. On the other hand, since $H \preceq G^{j}_{i_{j}}$, by Proposition 2.1, we have $\mu_{H}(g^{j}_{i_{j}}) \ge \mu_{G^{j}_{i_{j}}}(g^{j}_{i_{j}})$. So, $\mu_{G^{j}}(g^{j}_{i_{j}})+\mu_{H}(g^{j}_{i_{j}}) \ge \mu_{G^{j}}(g^{j}_{i_{j}})+\mu_{G^{j}_{i_{j}}}(g^{j}_{i_{j}}) >g^{j}_{i_{j}}$ and therefore $H \sim_{si} G^{j}$ at $g^{j}_{i_{j}}$. Thus, the family $\mathcal{G} \cup \mathcal{H}(\mathcal{G})$ is strongly intersecting. Theorem 1.1 gives us that $\mathcal{F}(\mathcal{G}) \cup \mathcal{F}(\mathcal{H}(\mathcal{G}))$ is LCIF. \\ \\
\textbf{ The family $\mathcal{F}(\mathcal{G})\cup \mathcal{F}(\mathcal{H}(\mathcal{G}))$ is maximal}\\
Next, we will show that $\mathcal{A}=\mathcal{F}(\mathcal{G}) \cup \mathcal{F}(\mathcal{H}(\mathcal{G}))$ is maximal. Assume that $\mathcal{A}\subsetneq \mathcal{B}$, where $\mathcal{B}$ is a LCIF. We take $S \in \mathcal{B} \setminus \mathcal{A}$. Since $S \notin \mathcal{A}$, it follows that $S \notin \mathcal{F}(\mathcal{G})$ and $S \notin \mathcal{F}(\mathcal{H}(\mathcal{G}))$.\\ 
\textbf{ The case of $1 \in S$ }.\\
Assume $S=\{1,s_{2},\ldots,s_{k}\}$. Since both $S$ and $\mathcal{F}(G^{j})$ are contained in $\mathcal{B}$, we conclude that $\mathcal{F}(S) \cup \mathcal{F}(G^{j})$ is intersecting. So, $S \sim_{si} G^{j}$. From Lemma 2.2, there exist $1\leq i_{j} \leq r_{j}$ and $1 \leq p \leq k$ such that $s_{p}=g^{j}_{i_{j}}=p+i_{j}-1$. It follows that $s_{2} \leq i_{j}+1$. Thus, $S \preceq \{1\} \cup [i_{j}+1,g^{j}_{i_{{j}}}]=G^{j}_{i_{j}}$ for all $j \in [m]$. Lemma 2.3 implies that $S \preceq G^{1}_{i_{1}} \wedge \cdots \wedge G^{m}_{i_{m}}$. It is obvious $G^{1}_{i_{1}} \wedge \cdots \wedge G^{m}_{i_{m}} \in \mathcal{H}(\mathcal{G})$. Thus, $S \in \mathcal{F}(\mathcal{H}(\mathcal{G}))$ which is a contradiction.\\ 
\noindent
\textbf{ The case $1 \notin S$}.\\ Since $S \notin \mathcal{F}(\mathcal{G})$, for each $G^{j} \in \mathcal{G}$, there exists $i_{j} \in [r_{j}]$ such that $s_{i_{j}} \geq g^{j}_{i_{j}}+1$. For such a tuple $(i_{1},\ldots,i_{m}) \in \left([r_{1}] \times \ldots [r_{m}]\right)$, we have $G=G^{1}_{i_{1}} \wedge \cdots \wedge G^{m}_{i_{m}} \in \mathcal{H}(\mathcal{G})$. For convenience, we simplify the structure of the set $G$ as follows. By Lemma 2.4, if among the sets $G^{1}_{i_{1}},\ldots,G^{m}_{i_{m}}$, there exist two sets $G^{p}_{i_{p}}$ and $G^{q}_{i_{q}}$ such that $G^{p}_{i_{p}} \preceq G^{q}_{i_{q}}$, then $G^{p}_{i_{p}} \wedge G^{q}_{i_{q}}= G^{p}_{i_{p}}$. So, we can remove $G^{q}_{i_{q}}$ without changing $G$. We consider the following possible cases. If $i_{p}=i_{q}$ then we must have $G^{p}_{i_{p}} \subseteq G^{q}_{i_{q}}$ or $G^{q}_{i_{q}} \subseteq G^{p}_{i_{p}}$. This means that $G^{p}_{i_{p}} \preceq G^{q}_{i_{q}}$ or $G^{q}_{i_{q}} \preceq G^{p}_{i_{p}}$. Hence, we can remove $A$ either $G^{p}_{i_{p}}$ or $G^{q}_{i_{q}}$ from the structure of $G$ without changing $G$. Next, consider $i_{p} < i_{q}$. If $g^{q}_{i_{q}} \le g^{p}_{i_{p}}$ then $G^{q}_{i_{q}} \subseteq G^{p}_{i_{o}}$ and so $G^{p}_{i_{p}} \preceq G^{q}_{i_{q}}$. In this case, we also remove $G^{q}_{i_{q}}$ from $G$ and $G$ remains unchanged. Finally, assume that there exist $p$ such that $g^{p+1}_{i_{p+1}} -g^{p}_{i_{p}} \le i_{p+1}-i_{p}$. We have $g^{p+1}_{i_{p+1}}-i_{p+1}+1 \le g^{p}_{i_{p}}-i_{p}+1$. This means $\left|G^{p+1}_{i_{p+1}}\right| \le \left|G^{p}_{i_{p}}\right|$. From the above analysis, we may regard that $i_{p} <i_{p+1}$, it follows that $G^{p}_{i_{p}} \preceq G^{p+1}_{i_{p+1}}$. Thus, we can also remove the set $G^{p+1}_{i_{p+1}}$ without changing $G$. 
From the above observation, we may assume (renumbering the indices if necessary) that  
\[
\begin{aligned}
&(1)\quad i_1 < i_2 < \cdots < i_p,\\
&(2)\quad \left|G^{j+1}_{i_{j+1}}\right| \ge \left|G^{j}_{i_{j}}\right|+1 
\quad\text{for all }0\le j\le p-1.
\end{aligned}
\]
Then, $G$ remains unchanged, that is $G=G^{1}_{i_{1}} \wedge \cdots \wedge G^{m}_{i_{m}}=G^{1}_{i_{1}} \wedge \cdots \wedge G^{p}_{i_{p}}$.\\
It is easy to see that $\left|G^{j}_{i_{j}}\right|+i_{j}-1 < \left|G^{j}_{i_{j}}\right|+i_{j+1}$. So, all sets of the form 
$\left[\left|G^{j}_{i_{j}}\right|+i_{j+1},\left|G^{j+1}_{i_{j+1}}\right|+i_{j+1}-1\right]$ are pairwise disjoint for all $0 \le j \le p-1$.
Now, we define  
$$G'=\{1\} \cup \bigcup_{j=0}^{p -1}\left[\left|G^{j}_{i_{j}}\right|+i_{j+1},\left|G^{j+1}_{i_{j+1}}\right|+i_{j+1}-1\right]$$ with the convention that $g^{0}_{i_{0}}=i_{j_{0}}=\left|G^{0}_{i_{0}}\right|=1$.\\
We will show that $G=G'$. Recall that the size of $G$ is $d=max\{g^{j}_{i_{j}}-i_{j}+1: j \in [p] \}$. From condition (2), we deduce that $d=g^{p}_{i_{p}}-i_{p}+1=|G^{p}_{i_{p}}|$. On the other hand, the size of $\left[\left|G^{j}_{i_{j}}\right|+i_{j+1},\left|G^{j+1}_{i_{j+1}}\right|+i_{j+1}-1\right]$ is $\left|G^{j+1}_{i_{j+1}}\right|-\left|G^{j}_{i_{j}}\right|$. So, the size of $G'$ is $\left|G^{p}_{i_{p}}\right|-\left|G^{0}_{i_{0}}\right|=\left|G^{p}_{i_{p}}\right|$. Hence, $|G|=|G'|$. It remains to show that the $i$-coordinate of $G$ coincides with that of $G'$. Assume $G=\{g_{1},\ldots,g_{d}\}$ and $G'=\{g'_{1},\ldots,g'_{d}\}$. For all $i \in [d]$,  there exists a unique index \(q\) such that $G^{q}_{i_{q}}+1 \le i \le G^{q+1}_{i_{q+1}}$. We observe that all elements of the set $\left[\left|G^{q}_{i_{q}}\right|+i_{q+1},\left|G^{q+1}_{i_{q+1}}\right|+i_{q+1}-1\right]$ are consecutive integers starting at  $\left(\left|G^{q}_{i_{q}}\right|+1\right)$-th coordinate and ending at the $\left(\left|G^{q+1}_{i_{q+1}}\right|\right)$-th coordinate of $G'$. Thus, $g'_{i}= i+i_{q+1}-1$. Next, we consider the $i$-th coordinate of each $G^{j}_{i_{j}}$ with $j \in [p]$. Since $G^{j}_{i_{j}}=\{1\} \cup \left[i_{j}+1,g^{j}_{i_{j}}\right]$, the $i$-th coordinate of each $G^{j}_{i_{j}}$ is $i+i_{j}-1$ if $i \in \left[2,\left|G^{j}_{i_{j}}\right|\right]$ and $+\infty$ if $i > \left|G^{j}_{i_{j}}\right|$. Since $G=G^{1}_{i_{1}}\wedge \cdots \wedge G^{p}_{i_{p}}$, we have $g_{i} = min\{i+i_{j}-1: j \in [p]\}$. Since $G^{q}_{i_{q}}+1 \le i \le G^{q+1}_{i_{q+1}}$, all the $i$-th coordinates of $G^{j}_{i_{j}}$ for $j \in [q]$ are equal to $+\infty$. Therefore, the minimum value of \(i + i_{j} - 1\) over all \(j \in [p]\) is exactly \(i + i_{q+1} - 1\). Thus, the $i$-th coordinate of $G$ is $g_{i}=i+i_{q+1}-1$, which coincides with that of $G'$. Hence, we conclude that $G=G'$.\\ 
We will prove that $G$ and $S$ are not intersecting, that is, for all $1 \leq \ell \leq n$, we always have $\mu_{S}(\ell)+\mu_{G}(\ell) \leq \ell$. For $\ell = 1$, we have $\mu_G(1) = \mu_{G^{j}_{i_j}}(1) = 1 \text{ for all } j \in [p].$ Next, consider $ \left|G^{j}_{i_{j}}\right|+i_{j+1} \le \ell \le  \left|G^{j+1}_{i_{j+1}}\right|+i_{j+1}-1$. Then
$\mu_{G}(\ell) = \ell - \left|G^{j}_{i_{j}}\right| - i_{j+1} + 1 + \left|G^{j}_{i_{j}}\right| = \ell - i_{j+1} + 1$, and it is straightforward to see that
$\mu_{G^{j+1}_{i_{j+1}}}(\ell) = \ell - i_{j+1} + 1 = \mu_{G}(\ell)$.
Next, for $\left|G^{j}_{i_{j}}\right|+i_{j+1} \le \ell \le  \left|G^{j}_{i_{j}}\right|+i_{j+1}-1|$, we have $\ell \notin G$. Hence, $\mu_{G}(\ell) = \mu_{G^{j}_{i_{j}}}(\ell) = g^{j}_{i_{j}} - i_{j} + 1$.
Finally, for $\ell > g^{p}_{i_{p}}$, we have
$\mu_{G}(\ell) = \mu_{G^{p}_{i_{p}}}(\ell) = g^{p}_{i_{p}} - i_{p} + 1$. Therefore, in all cases, for each $1 \le \ell \le n$, there exists $j \in [p]$ such that $\mu_{G}(\ell) = \mu_{G^{j}_{i_{j}}}(\ell)$. By Lemma 2.3, $S$ and $G^{j}_{i_{j}}$ are not strongly intersecting with respect to any $G^{j}_{i_{j}}$. Consequently, for all $1 \le \ell \le n$ we have,
$\mu_{S}(\ell) + \mu_{G}(\ell) = \mu_{S}(\ell) + \mu_{G^{j}_{i_{j}}}(l) \le l$.\\
In summary, we have shown that $\mu_{S}(\ell)+\mu_{G}(\ell) \leq \ell$ for all $1 \leq \ell \leq n$. Hence, $S$ and $G$ are not strongly intersecting. This means there exists two sets $T \leq S$ and $H \leq G$ such that $T \cap H=\emptyset$. It is obvious that $T \in \mathcal{A}$ and $H \in \mathcal{A}$ ( since $H \in \mathcal{F}(\mathcal{H}(\mathcal{G}))$) and we have a contradiction. \\ \\
\textbf{ Every MLCIF $\mathcal{A} \subseteq \binom{[n]}{k}$, there exists PGS $\mathcal{G}$ such that $\mathcal{A}=\mathcal{F}(\mathcal{G})\cup \mathcal{F}(\mathcal{H}(\mathcal{G}))$}\\
Since $\mathcal{A}$ is left-compressed, we can completely enumerate all elements of $\mathcal{A}$ through its maximal elements. Let $\mathcal{A}^{*}$ be the set of all maximal elements of $\mathcal{A}$ with respect to the order ``$\le$''. For each $A = \{a_{1},\ldots,a_{k}\} \in \mathcal{A}^{*}$, we have $a_{1} < k+1$, since otherwise there exists $B = \{1,\ldots,k\} \le A$ and $A \cap B = \emptyset$. Let $r = r_{A}$ be the largest index such that $a_{r} < r + k$. We define $\pi(A)=\{a_{1},\ldots,a_{r}\}$ and $\pi(\mathcal{A}^{*}) = \{\pi(A) : A \in \mathcal{A}^{*}\}$. Note that $A \preceq \pi(A)$ for all $A \in \mathcal{A}^{*}$.

We now show that $\pi(\mathcal{A}^{*})$ is strongly intersecting. Take any $G, H \in \pi(\mathcal{A}^{*})$. Then $G = \pi(A)$ and $H = \pi(B)$ for some $A, B \in \mathcal{A}^{*}$. Since $A \sim_{si} B$, there exists $\ell \in [n]$ such that $\mu_{A}(\ell) + \mu_{B}(\ell) > \ell$. Assume $A=\{a_{1},\ldots,a_{k}\}$ and $G=\pi(A)=\{a_{1},\ldots,a_{r}\}$. We will prove that $\ell \le k+r-1$.

First, consider the case $r = k$. Then $a_{k} < 2k$. For $\ell \ge k+r = 2k$, we have $\mu_{G}(\ell)=k$, and hence $\mu_{G}(\ell)+\mu_{H}(\ell) \le 2k \le \ell$, a contradiction. Thus $\ell \le k+r-1$.

Next, consider the case $r < k$. In this case, $a_{r+t} \ge k+r+t$ for all $1 \le t \le k-r$, and in particular $a_{k} \ge 2k$. Now consider any $\ell \ge k+r$.
\begin{itemize}
    \item If $\ell \ge a_{k}$, then $\mu_{G}(\ell)=k$, and so $\mu_{G}(\ell)+\mu_{H}(\ell)\le 2k\le \ell$.
    \item If $\ell = k+r$, then $\mu_{G}(\ell)=r=\ell-k$, and thus $\mu_{G}(\ell)+\mu_{H}(\ell)\le \ell$.
    \item If $k+r+1 \le \ell < a_{k}$, then there exists $r+1 \le i \le k-1$ such that $a_{i} \le \ell < a_{i+1}$. Since $i \ge r+1$, we have $a_{i} \ge k+i$, so $\mu_{G}(\ell)=i \le a_{i}-k$, again giving $\mu_{G}(\ell)+\mu_{H}(\ell)\le \ell$.
\end{itemize}
Thus, every case with $\ell \ge k+r$ yields a contradiction, and therefore $\ell \le k+r-1$. For such $\ell$ we have $\mu_{G}(\ell)=\mu_{A}(\ell)$ and similarly $\mu_{H}(\ell)=\mu_{B}(\ell)$, so $\mu_{G}(\ell)+\mu_{H}(\ell)=\mu_{A}(\ell)+\mu_{B}(\ell)>\ell$. Hence $G \sim_{si} H$ at $\ell$, proving that $\pi(\mathcal{A}^{*})$ is strongly intersecting.

Now define two families $\mathcal{G}=\pi(\mathcal{A}^{*})\cap \mathcal{G}_{k}(\bar{1})$ and $\mathcal{G}'=\pi(\mathcal{A}^{*})\cap \mathcal{G}_{k}(1)$. Observe that $\mathcal{G}$ is a PGS. Suppose $\mathcal{G}=\{G^{1},\ldots,G^{m}\}$. We prove that
\[
\mathcal{A} \subseteq \mathcal{F}(\mathcal{G}) \cup \mathcal{F}(\mathcal{H}(\mathcal{G})).
\]
Indeed, take any $S \in \mathcal{A}$. Then there exists $A \in \mathcal{A}^{*}$ such that $S \le A$. Clearly, $S \preceq \pi(A)$. If $1 \notin S$, then $\pi(A) \in \mathcal{G}$ and hence $S \in \mathcal{F}(\mathcal{G})$. If $1 \in S$, then $\pi(A) \in \mathcal{G}'$ and thus $S \in \mathcal{F}(\mathcal{G}')$. Therefore,
\[
\mathcal{A} \subseteq \mathcal{F}(\mathcal{G}) \cup \mathcal{F}(\mathcal{G}').
\]

Next, we prove that $\mathcal{G}' \subseteq \mathcal{H}(\mathcal{G})$. Take any $H \in \mathcal{G}'$. For every $G^{j} \in \mathcal{G}$, we have $H \sim_{si} G^{j}$. By Lemma 2.2, there exist indices $p_{j}$ and $i_{j}$ such that $h_{p_{j}}=g^{j}_{i_{j}}=p_{j}+i_{j}-1$, which implies $H \preceq G^{j}_{i_{j}}$ at $\ell=p_{j}+i_{j}-1$ for all $1\le j\le m$. By Lemma 2.3,
\[
H \preceq G^{1}_{i_{1}}\wedge\cdots\wedge G^{m}_{i_{m}}.
\]
Let $G=G^{1}_{i_{1}}\wedge\cdots\wedge G^{m}_{i_{m}}$. Clearly $|H| \ge |G|$. If $|H|>|G|$, then replacing $H$ by a proper prefix would enlarge $\mathcal{A}$, contradicting its maximality. Hence $|H|=|G|$ and therefore $H \le G(i_{1},\ldots,i_{m})$, showing $H \in \mathcal{H}(\mathcal{G})$.

Since both $\mathcal{F}(\mathcal{G})\cup\mathcal{F}(\mathcal{H}(\mathcal{G}))$ and $\mathcal{F}(\mathcal{G})\cup\mathcal{F}(\mathcal{G}')$ are maximal LCIFs, and $\mathcal{G}' \subseteq \mathcal{H}(\mathcal{G})$, it follows that $\mathcal{G}'=\mathcal{H}(\mathcal{G})$. Therefore,
\[
\mathcal{A}=\mathcal{F}(\mathcal{G})\cup\mathcal{F}(\mathcal{H}(\mathcal{G})).
\]

Thus, Theorem 1.3 is proved.

\end{proof}

\section{ Proofs of Theorem 1.5 and Theorem 1.6}
\subsection{ Proof of Theorem 1.5} \mbox{}\\
It is known that every trivial LCIF is contained in Star family. Therefore, in this section we only consider non-trivial families. In what follows, we determine all MLCIFs with two maximal generators, in particular those of rank 2.
\begin{proof} The generating set of $\mathcal{A}$ is $\mathcal{G}\cup \mathcal{H}(\mathcal{G})$. From assumption, we have the family $\mathcal{G}\cup \mathcal{H}(\mathcal{G})$ has exxactly two maximal generators. Since $\mathcal{A}$ is a non-trivial MLCIF, we must have $|\mathcal{G}| > 0$. Thus, each of $G$ and  $\mathcal{H}$ has exactly a maximal generator. We suppose that maximal generator of $\mathcal{G}$ is $G=\{g_{1},\ldots,g_{r}\}$. Since $G \in \mathcal{G}_{k}(\overline{1})$ and $\mathcal{A}$ is maximal, by Lemma 2.3, we have $g_{1} \geq 2$ and $g_{r}\leq r+k-1$. If  $r=1$, then $G=\{a\}$, where $a\ge 2$. However, in this case, $G$ is not strongly intersecting. Therefore, we must have $r \ge 2$. The generators of the family $\mathcal{H}(G)$ are $G_{i}=\{1\} \cup [i+1,g_{i}]$ for $1 \leq i \leq r$. Assume that $g_{i+1} \geq g_{i}+2$ for some $1 \leq i \leq r-1$, then $G_{i}=\{1\}\cup[i+1,g_{i}]$ and $G_{i+1}=\{1\}\cup[i+2,g_{i+1}]$ are maximal elements of $\mathcal{H}(G)$. This contradicts the fact that $\mathcal{H}(\mathcal{G})$ has exactly one maximal generator. Therefore, the elements of $G$ are consecutive natural numbers. This means $G=[a,b]$, where $a$ and $b$ are natural numbers and $b \ge a+1$. If $b \geq a+k$ then $|G|=b-a+1 \geq k +1$, which is a contradiction since $G$ is a generator. So, we must have $b \leq a+k-1$. We will prove that $b > 2a-1$. Indees, first we assume that $b=2a-1$.
Since $|\mathcal{H}(\mathcal{G})|=1$, Then we have $G=[a,2a-1]$ and $H=\{1\} \cup [a+1,2a-1]$. We have $|G|=|H|=a$ and $H \leq G$. So $H \preceq G$. We see that the unique maximal element of $\mathcal{A}$ is $G$, contrary to the assumption. Next, we asume $b < 2a-1$. The number of elements of $G$ is $r=b-a+1 \leq a-1$. For all $1 \leq i \leq r$, we have $g_{i}=a+i-1=(a-1)+i \geq r+i \geq 2i$. However, this contradicts the fact that $G$ is strongly intersceting. In summary, we have proved that $G=[a,b]$ with $b > 2a-1$. From this, we see that the unique maximal generator of $\mathcal{H}(\mathcal{G})$ is $H=\{1\} \cup [b-a+2,b]$. Therefore, we get $\mathcal{A}=\mathcal{F}([a,b])\cup \mathcal{F}(\{1\}\cup[b-a+2,b])$.\\
Now, suppose that $\mathcal{A}$ has rank 2. Since $|G|=b-a+1 > a = |H|$, we must have $a=2$. From this, we have $b \leq k+1$. On the other hand, if $b=3$ then we have the family $\mathcal{A}_{2,3}$ which contains the unique maximal generator $\{2,3\}$. So, we must have $4 \leq b \leq k+1$. In summary, all MLCIFs having two maximal generators with rank 2 are $\mathcal{A}=\mathcal{F}([2,b])\cup \mathcal{F}(\{1,b\})$ with $4 \leq b \leq k+1$.
\end{proof}
The $\mathcal{AHM}_{b}$ families \cite{B} are those with two maximal generators and rank 2. That is, they are all represented by $\mathcal{A}=\mathcal{F}([2,b])\cup \mathcal{F}(1,b)$. Specially, if $b=k+1$ then we have the Hilton-Milner family. 

\subsection{ Proof of Theorem 1.6 }
To prove Theorem 1.6, we need several calculations given in the following lemmas.\\
For $G \in 2^{[n]}$, recall that $\mathcal{F}(G)=\{ S \in  \binom{[n]}{k}: S \preceq G \}$ and $\mathcal{L}(G)=\{S \in \binom{[n]}{|G|}: S \leq G\}$. For $G=\{g_{1},\ldots,g_{r}\}$, we also write $\mathcal{F}(G)=\mathcal{F}(g_{1},\ldots,g_{r})$ and $\mathcal{L}(G)=\mathcal{L}(g_{1},\ldots,g_{r})$. In the following, we establish some formulas concerning the size of the family $\mathcal{F}(G)$ and $\mathcal{L}(G)$. 
First, we have the following formulas, whose proofs are completely elementary.
\begin{lem}
Let $G \in 2^{[n]}$ with $G = \{ g_{1}, g_{2}, \ldots, g_{r} \}$. Then:
\begin{enumerate}[(i)]
    \item $|\mathcal{L}(g_{1}, \ldots, g_{r})|
        = \sum_{i=1}^{g_{1}} 
          |\mathcal{L}(g_{2}-i, \ldots, g_{r}-i)|$;
    \item $|\mathcal{L}([a,b])|
        = |\mathcal{L}(\{a, a+1, \ldots, b\})|
        = \binom{b}{a-1}$, for $a \le b$.
\end{enumerate}
\end{lem}
\begin{proof}(i) We partition the set $\mathcal{L}(G)$ into subsets $\mathcal{L}_{i}$ as follows:$$\mathcal{L}_{i}=\left \{A=\{a_{1},\ldots,a_{r}\} \in \mathcal{L}(G): a_{1}=i\right \}, 1 \le i \le g_{1}$$
\[
\mathcal{L}(g_{1},\dots,g_{r})=\bigsqcup_{i=1}^{g_{1}}\mathcal{L}_{i},
\qquad
\mathcal{L}_{i}\cap\mathcal{L}_{j}=\varnothing\quad\text{for all }1\le i<j\le g_{1}.
\]
Note that $a_{j}-i \le g_{j}-i$, so the set $\{a_{2}-i,\ldots,a_{r}-i\} \le \{g_{2}-i,\ldots,g_{r}-i\}$, for all $2 \le j \le r$. It implies that $\{a_{2}-i,\ldots,a_{r}-i\} \in \mathcal{L}(g_{2}-i,\ldots,g_{r-i}\}$. Thus, we obtain a bijection between $\mathcal{L}_{i}$ and $\mathcal{L}(g_{2}-i,\ldots,g_{r}-i)$. Hence $|\mathcal{L}(g_{1}, \ldots, g_{r})|
        =\sum_{i}^{g_{1}}|\mathcal{L}_{i}|= \sum_{i=1}^{g_{1}} 
          |\mathcal{L}(g_{2}-i, \ldots, g_{r}-i)|$.\\
(ii) Each set $S \in \mathcal{L}([a,b])$
corresponds to a unique way of choosing a subset of $[b]$ consisting of $b-a+1$ elements. So, $|\mathcal{L}([a,b])|
        = |\mathcal{L}(\{a, a+1, \ldots, b\})|
        = \binom{b}{b-a+1}=\binom{b}{a-1}$.
\end{proof}
\begin{lem}
If $0 \le a < b \le \left\lfloor \frac{n}{2} \right\rfloor$, then $\binom{n}{a} < \binom{n}{b}$.
\end{lem}
\begin{proof}
We count the number of pairs $(A,B)$ satisfying $ A \subseteq B \subseteq [n], |A|=a$ and $|B|=b$ in two different ways. Choose $B$ first, which can be done in $\binom{n}{b}$ ways. Then choose $A$ as an $a$-subset of $B$, which can be done in $\binom{b}{a}$ ways. Hence, the total number of such pairs $(A,B)$ is $\binom{n}{b}\binom{b}{a}$. Next, choose $A$ first, which can be done in $\binom{n}{a}$ ways. Then choose $b-a$ elements from $[n]\setminus A$ to form $B$, which can be done in $\binom{n-a}{\,b-a\,}$ ways. Therefore, the total number of pairs is also $\binom{n}{a}\binom{n-a}{\,b-a\,}$.
So 
\[
\binom{n}{b}\binom{b}{a} = \binom{n}{a}\binom{n-a}{\,b-a\,}.
\]
Thus,
\[
\frac{\binom{n}{b}}{\binom{n}{a}} = \frac{\binom{n-a}{\,b-a\,}}{\binom{b}{a}}.
\]

Since $0 \le a < b \le \left\lfloor \frac{n}{2} \right\rfloor$, we have $n-a > b$. Therefore,
\[
\binom{n-a}{\,b-a\,} > \binom{b}{\,b-a\,} = \binom{b}{a}.
\]
which implies $\binom{n}{b} > \binom{n}{a}$. This completes the proof.
\end{proof}

To prove Theorem 1.6, we need to determine the size of a rank-$2$ MLCIF that has exactly two maximal generators. The following lemma provides this.
\begin{lem} Let $\mathcal{A}=\mathcal{F}([2,b]) \cup \mathcal{F}(1,b)$, where $4 \leq b \leq k+1$. Then, 
  $$|\mathcal{A}|=\binom{n-1}{k-1}-\binom{n-b}{k-1}+\binom{n-b}{k-b+1}$$
\end{lem}
\begin{proof} We will use the formula $$|\mathcal{A}|=\left|\mathcal{F}([2,b])\right|+|\mathcal{F}(1,b)|-|\mathcal{F}([2,b])\cap \mathcal{F}(1,b)|$$
By applying Lemma 3.1 and noting that $|\mathcal{F}(g_{1},\ldots,g_{r})|=|\mathcal{L}(\{g_{1},\ldots,g_{r}\}\cup [n-k+r+1,n])|$, we obtain:
\begin{align*}
|\mathcal{F}(1,b)|&=\left|\mathcal{L}\left(\{1,b\}\cup[n-k+3,n]\right)\right|\\
&=|\mathcal{L}(\{b-1\}\cup [n-k+2,n-1])| =\sum_{i=1}^{b-1}|\mathcal{L}([n-k-i+2,n-i-1])|\\
&=\sum_{i=1}^{b-1}\binom{n-i-1}{n-k-i+1}=\sum_{i=1}^{b-1}\binom{n-i-1}{k-2}\\
&=\binom{n-2}{k-2}+\binom{n-3}{k-2}+\ldots+\binom{n-b}{k-2}\\&=\binom{n-1}{k-1}-\binom{n-b}{k-1}
\end{align*}
Next, we we will compute $\left|\mathcal{F}([2,b])\right|$. Assume $S \in \mathcal{F}([2,b])$ with $S=\{s_{1},\ldots,s_{b},\ldots,s_{k}\} \preceq \{2,\ldots,b\}$. Clearly we cannot have $s_{b} \leq b-1$, So, there are two possibilities for $s_{b}$. First, we consider $s_{b}\geq b+1$. We have $\{s_{1},\ldots,s_{b-1}\} \leq \{2,\ldots,b\}$. It implies that $\{s_{1},\ldots,s_{b-1}\} \in \binom{[b]}{b-1}$ and $\{s_{b},\ldots,s_{k}\} \in \binom{[b+1,n]}{k-b+1}$, Thus, the number of sets $S$ in this case is $b\binom{n-b}{k-b+1}$. Next, we consider $s_{b}=b$. We must have $\{s_{1},\ldots,s_{b-1}\}=\{1,\ldots,b-1\}$ and $\{s_{b+1},\ldots,s_{k}\} \in \binom{[b+1,n]}{k-b}$. We see that the number of sets $S$ in this case is $\binom{n-b}{k-b}$. Thus, $$|\mathcal{F}([2,b])|=b\binom{n-b}{k-b+1}+\binom{n-b}{k-b}$$.\\
Proceeding in the same way, we compute $|\mathcal{F}([2,b])\cap \mathcal{F}(1,b)|$. It is easy to prove that
$\mathcal{F}([2,b])\cap \mathcal{F}(1,b)= \mathcal{F}(\{1\}\cup [3,b])$. Take $S \in \mathcal{F}(\{1\}\cup[3,b])$. Assume $S=\{1,s_{2},\ldots,s_{k}\}$. We have $\{1,s_{2},\ldots,s_{b-1}\} \leq \{1,3,\ldots,b\}$. If $s_{b} \le b-1$ then $s_{2} \le 1$. This is a contradiction. So, we must have $s_{b} \geq b$. We consider two cases. First, we consider $s_{b} \ge b+1$. We have $\{s_{2},\ldots,s_{b-1}\} \subseteq [2,b]$ so there are $b-1$ such sets $\{s_{2},\ldots,s_{b-1}\}$ in total. There are also $\binom{n-b}{k-b+1}$ ways to choose the sets $\{s_{b},\ldots,s_{k}\} \subseteq [b+1,n]$. Hence, there are $(b-1)\binom{n-b}{k-b+1}$ set $S$ satisfying conditions $S \preceq \{1,3,\ldots,b\}$ and $s_{b} \geq b+1$. We consider $s_{b}=b$. We have $\{1,s_{2},\ldots,s_{b-1}\}=[b-1]$. We have $\{s_{b+1},\ldots.s_{k}\} \subseteq [b+1,n]$. Thus, there are $\binom{n-b}{k-b}$ ways to choose $\{s_{b+1},\ldots.s_{k}\}$. Hence, there are $\binom{[n-b]}{k-b}$ ways to choose $S \in \binom{[n]}{k}$ satisfying the conditions $S \preceq \{1,3,\ldots,b\}$ and $s_{b}=b$. In summary, $$\left|\mathcal{F}([2,b])\cap \mathcal{F}(1,b)\right|=\left|\mathcal{F}(\{1\}\cup[3,b])\right|=(b-1)\binom{n-b}{k-b+1}+\binom{n-b}{k-b}$$
Finally, we get 
$$|\mathcal{A}|=\binom{n-1}{k-1}-\binom{n-b}{k-1}+\binom{n-b}{k-b+1}$$
\end{proof}
\noindent
\textbf{Proof Theorem 1.6}
\begin{proof}
\noindent
(i)Assume $X \cap [2,b] \neq \emptyset$. Let $Y = [n] \setminus X$. Then $1 \in Y$ and $|Y| = n - d$. We will use the formula $|\mathcal{A}(X)|=|\mathcal{F}(1,b)(X)|+|(\mathcal{F}([2,b])\setminus\mathcal{F}(1,b))(X)|$
We first count the number of members of $\mathcal{F}(1,b)$ that are disjoint from $X$.
Note that every element of $[b]$ belongs to exactly one of the two sets $X$ or $Y$.
The number of elements of $X$ not exceeding $b$ is $|X \cap [b]| = \mu_X(b)$.
Hence, the number of elements of $Y$ not exceeding $b$ is $s = b - \mu_X(b)$.

Let $S$ be a $k$-set in $\mathcal{F}(1,b)$ such that $S \cap X = \emptyset$.
Then $S \in \mathcal{F}(1,b) \cap \binom{Y}{k}$.
Since $S \preceq \{1,b\}$, we must have $1 \in S$.
Let $T = S \setminus \{1\}$. Clearly, $T \in \binom{Y \setminus \{1\}}{k-1}$.

If $T \in \binom{Y \setminus \{1\}}{k-1}$, then either $\{1\} \cup T \in \mathcal{F}(1,b)$ or $\{1\} \cup T \notin \mathcal{F}(1,b)$.
In the latter case, $T$ must belong to $\binom{Y \setminus \{y_1, \dots, y_s\}}{k-1}$.
Therefore, the number of $(k-1)$-sets in $\binom{Y \setminus \{1\}}{k-1}$ not contained in $\mathcal{F}(1,b)$ is $\binom{n - d - s}{k-1}$.
Hence, the number of those contained in $\mathcal{F}(1,b)$ is
\[
\binom{n - d - 1}{k - 1} - \binom{n - d - s}{k - 1}
= \binom{n - d - 1}{k - 1} - \binom{n - d - b + \mu_X(b)}{k - 1}.
\]

On the other hand, for any $S \in \mathcal{F}([2,b]) \setminus \mathcal{F}(1,b)$, we have $1 \notin S$, hence $[2,b] \subseteq S$.
Since $X \cap [2,b] \ne \emptyset$, we also have $S \cap X \ne \emptyset$.
Thus, the number of elements of $\mathcal{A}$ intersecting $X$ is
\[
|\mathcal{A}(X)| = |\mathcal{A}| - \binom{n - d - 1}{k - 1} + \binom{n - d - b + \mu_X(b)}{k - 1},
\]
that is, ( by Lemma 3.3)
\[
|\mathcal{A}(X)| = 
\binom{n - 1}{k - 1} - \binom{n - b}{k - 1}
+ \binom{n - b}{k - b + 1}
- \binom{n - d - 1}{k - 1}
+ \binom{n - d - b + \mu_X(b)}{k - 1}.
\]
We consider the special cases listed below. 
First, we consider the case $X \subseteq [2,b]$. All elements of $X$ are at most $b$; in the words $\mu_{X}(b)=d$. It follows that $$|\mathcal{A}(X)|=\binom{n-1}{k-1}-\binom{n-d-1}{k-1}+\binom{n-b}{k-b+1}=|\mathcal{S}(X)|+\binom{n-b}{k-b+1}$$
Next, we consider the case $X \cap [2,b] \subsetneq X$. In this case, we have $\mu_{X}(b) < d $. So, $d-\mu_{X}(b) \geq 1$. Thus,
$$\binom{n-d-b+\mu_{X}(b)}{k-1} \leq \binom{n-b-1}{k-1}=\binom{n-b}{k-1}-\binom{n-b-1}{k-2}$$
From this, we get
\begin{align*}
     |\mathcal{A}(X)|& = \binom{n-1}{k-1}-\binom{n-b}{k-1} +\binom{n-b}{k-b+1}-\binom{n-d-1}{k-1}+\binom{n-d-b+\mu_{X}(b)}{k-1}\\
     & \leq |\mathcal{S}(X)| +\binom{n-b}{k-b+1}-\binom{n-b-1}{k-2}
\end{align*}
Now, we prove that if $4 \le b \le k+1$ and $n$ is sufficiently large, then $\binom{n-b}{k-b+1} < \binom{n-b-1}{k-2}$.\\
Indeed, consider the ratio
\begin{align*}
R &= \frac{\binom{n-b}{\,k-b+1\,}}{\binom{n-b-1}{\,k-2\,}}= \frac{(n-b)!}{(k-b+1)!(n-k-1)!} \cdot \frac{(k-2)!(n-k-b+1)!}{(n-b-1)!}\\
&= (n-b)\cdot \frac{(k-2)!}{(k-b+1)!}\cdot \frac{(n-k-b+1)!}{(n-k-1)!}\\
&=(n-b).\prod_{i=1}^{b-3}(k-b+1+i).\frac{1}{\prod_{i=1}^{b-2}(n-k-b+1+i)}\\
&=\frac{n-b}{n-k-1}\cdot \frac{\prod_{i=1}^{b-3}(k-b+i+1)}{\prod_{i=1}^{b-3}(n-k-b+i+1)}=\frac{n-b}{n-k-1}.\prod_{i=1}^{b-3}\frac{k-b+i+1}{n-k-b+i+1}
\end{align*}
We find that for $n$ is sufficiently large, we have $\frac{k-b+i+1}{n-k-b+i+1} < 1$
On the other hand, consider $g(n)=\frac{n-b}{n-k-1}$. It is easy that $g(n)$ is an decreasing function. So, for sufficiently large $n$, we have $g(n) < 1 $. Thus, $R < 1$ for suficiently large $n$. It follows that $\binom{n-b}{k-b+1}-\binom{n-b-1}{k-2} < 0$ and therefore $\mathcal{A}(X) < \mathcal{S}(X)$.\\ 
\noindent
(ii) Assume $X \cap [2,b] = \emptyset$. We will count the number of sets in $\mathcal{A}$ that are disjoint from $X$. We will use the formula $|\mathcal{A}_{0}(X)|=|\mathcal{F}(1,b)_{0}(X)|+|(\mathcal{F}([2,b])_{0}(X)|-|\mathcal{F}(1,b)_{0}(X)\cap \mathcal{F}([2,b])_{0}(X)|$. First, we will count the number of elements of the set $\mathcal{F}(1,b)_{0}(X)$. 
The total number of $(k-1)$-sets of $[n - d - 1]$ is $\binom{n - d - 1}{k - 1}$, and among them, the number of elements that are disjoint from $X$ and not contained in $\mathcal{F}(1,b)$ is $\binom{n - d - b}{k - 1}$.
Hence,
$|\mathcal{F}(1,b)_{0}(X)|=\binom{n - d - 1}{k - 1} - \binom{n - d - b}{k - 1}$. Next, we count the sets $S \in \mathcal{F}([2,b])$ with $S \cap X = \emptyset$.
We may choose $A \le [2,b]$ in $b$ ways and a $(k - b + 1)$-set $B$ from $[b + 1, n] \setminus X$ in $\binom{n - b - d}{k - b + 1}$ ways.
Therefore, there are $b \binom{n - b - d}{k - b + 1}$ ways to form $S = A \cup B$ satisfying the required condition. Thus, $|\mathcal{F}([2,b])_{0}(X)|=b\binom{n-b-d}{k-1}$. Finally, we count $|\mathcal{F}(1,b)_{0}(X)\cap \mathcal{F}([2,b])_{0}(X)|$. It is obvious that $\mathcal{F}(1,b)_{0}(X)\cap \mathcal{F}([2,b])_{0}(X)=(\mathcal{F}(1,b) \cap \mathcal{F}([2,b]))_{0}(X)=\mathcal{F}(\{1\}\cup[3,b])_{0}(X)$.
Each $S \in \mathcal{F}(\{1\}\cup[3,b])_{0}(X)$ is obtained by choosing a set $A \le \{1\} \cup [3,b]$ and a set $B \in \binom{[b + 1, n]}{k - b + 1}$, and setting $S = A \cup B$. Therefore, $|\mathcal{F}(1,b)_{0}(X)\cap \mathcal{F}([2,b])_{0}(X)|=(b - 1)\binom{n - b - d}{k - b + 1}$.\\

Hence, the total number of elements of $\mathcal{A}$ that are disjoint from $X$ is
\begin{align*}
|\mathcal{A}_{0}(X)|&=\binom{n - d - 1}{k - 1} - \binom{n - d - b}{k - 1} 
   + b \binom{n - b - d}{k - b + 1} 
   - (b - 1)\binom{n - b - d}{k - b + 1} \\
&= \binom{n - d - 1}{k - 1} 
   - \binom{n - d - b}{k - 1}
   + \binom{n - b - d}{k - b + 1}.
\end{align*}
Consequently, the number of elements of $\mathcal{A}$ having nonempty intersection with $X$ is
\begin{align*}
|\mathcal{A}(X)| &=|\mathcal{A}|-|\mathcal{A}_{0}(X)|\\
&=\binom{n - 1}{k - 1} - \binom{n - b}{k - 1}
+ \binom{n - b}{k - b + 1}
- \binom{n - d - 1}{k - 1}
+ \binom{n - d - b}{k - 1}
- \binom{n - b - d}{k - b + 1}.
\end{align*}
\noindent
To establish $|\mathcal{A}(X)| \le |\mathcal{S}(X)|$, we rewrite the above equality as follows.
$$|\mathcal{A}(X)|= |\mathcal{S}(X)|+\left[\binom{n-b}{k-b+1}-\binom{n-b-d}{k-b+1}\right]-\left[\binom{n-b}{k-1}-\binom{n-b-d}{k-1}\right]$$
Using (1.3) we have 
$$\alpha=\binom{n-b}{k-b+1}-\binom{n-b-d}{k-b+1}= \sum_{j=0}^{d}\binom{n-b-d+j}{k-b}$$
and $$\beta=\binom{n-b}{k-1}-\binom{n-b-d}{k-1}=\sum_{j=0}^{d}\binom{n-b-d+j}{k-2}$$
To determine the sign of the difference $\alpha-\beta$, we need the following result, which can be easily proved by a combinatorial argument: if $ 0 \leq a < b \leq \left[\frac{n}{2}\right]$, then $\binom{n}{a} < \binom{n}{b}$.\\
Thus, $$\alpha-\beta=\sum_{j=0}^{d}\left[\binom{n-b-d+j}{k-b}-\binom{n-b-d+j}{k-2}\right]$$
Since $j \leq d$ and $b \geq 4$, We have $ n\geq 2k \geq 2k+d-j + b-4$. Therefore $n-b-d+j \geq 2k-4$ and it follows that $$\frac{n-b-d+j}{2}\geq k-2 > k-b$$
From Lemma 3.2 we have $\alpha -\beta <0$. We conclude that $\mathcal{A}(X) \leq \mathcal{S}(X)$.
\end{proof}

\section{ Further Direction}
Some questions related to the generating sets $\mathcal{G}$ of MLCIF $\mathcal{A}=\mathcal{F}(\mathcal{G})$ is arised in this section.\\ \\
For $A=\mathcal{F}(\mathcal{G})$, if $\mathcal{A}$ has rank 1 or 2 then we have full described the forms of $\mathcal{A}$. So, a question is raised naturally as follows:\\
\textbf{Question 1.}\textit{ Identify all families $\mathcal{A}=\mathcal{F}(\mathcal{G})$ with rank $r \geq 3$}\\ \\
Theorem 1.6 determines $\mathcal{A}(X)$ when $\mathcal{A}$
is an MLCIF with exactly two maximal generators and of rank 2. We propose the analogous problem in which the family $\mathcal{A}$ has rank at least 3,\\
\textbf{ Question 2.} \textit{Let $X \subseteq [2,n]$ with $|X|=d$ and let $\mathcal{A}$ be an MLCIF with exactly two maximal generators of rank $r \geq 3$. Compare $\mathcal{A}(X)$ and $\mathcal{S}(X)$}

\section*{Statements and Declarations}
\textbf{Funding}\\
The author declares that no funds, grants, or other financial support were received for the research, authorship, or publication of this article.\\ \\
\medskip
\textbf{Competing Interests}\\  
The author declares that there are no financial or non-financial interests that could be perceived as directly or indirectly influencing the work reported in this paper.\\ \\
\medskip
\textbf{Author Contributions}\\  
The author is solely responsible for the conception of the problem, the development of the proofs, and the preparation of the manuscript.\\

\end{document}